\pdfoutput=1
\documentclass[12pt, a4paper]{amsart}

\numberwithin{equation}{section}

\usepackage{amsmath,amsfonts,amssymb}
\usepackage{mathtools}
\usepackage[utf8]{inputenc}
\usepackage[english]{babel}
\usepackage{tikz-cd}

\usepackage[margin=1.2in,footskip=0.25in]{geometry}

\DeclareUnicodeCharacter{FB01}{fi}
\newtheorem{theorem}{Theorem}[section]
\newtheorem{corollary}[theorem]{Corollary}
\newtheorem{lemma}[theorem]{Lemma}
\newtheorem{prop}[theorem]{Proposition}
\newtheorem{remark}[theorem]{Remark}

\theoremstyle{definition}
\newtheorem{definition}{Definition}[section]

\begin{document}
	\title[PICARD GROUP OF MODULI OF PARABOLIC HIGGS]{Picard Group of Moduli of Parabolic Higgs bundles}
	
\author{Sumit Roy}
\address{Center for Geometry and Physics, Institute for Basic Science (IBS), Pohang 37673, Korea}
\email{sumit@ibs.re.kr}
\thanks{E-mail : sumit@ibs.re.kr}
\thanks{Address: Center for Geometry and Physics, Institute for Basic Science (IBS), Pohang 37673, Korea}
	
\begin{abstract}
	Let $X$ be a compact Riemann surface of genus $g\geq 3$ and let $\mathcal{M}_{\mathrm{par}}$ be the moduli space of semistable parabolic bundles over $X$. Let $\mathcal{M}_{\mathrm{parH}}$	denote the moduli space of semistable parabolic Higgs bundles over $X$. In this article, we study the Picard group of $\mathcal{M}_{\mathrm{par}}$ and $\mathcal{M}_{\mathrm{parH}}$. 
\end{abstract}

\maketitle

\section{Introduction}
Let $X$ be a compact Riemann surface of genus $g\geq 3$ and let $D=\{p_1, \dots , p_n\}\subset X$ be a fixed set of points in $X$. The notion of parabolic bundles over a curve was described by Seshadri in \cite{S77}. Later, Mehta and Seshadri in \cite{MS80} constructed its moduli space $\mathcal{M}_{\mathrm{par}}$ using Geometric Invariant Theory. Also, this moduli space is a normal projective variety. Their motivation was to extend the Narasimhan-Seshadri correspondence in the case of irreducible unitary representations of the fundamental group of the Riemann surface with given finitely many punctures. A \textit{parabolic Higgs bundle} on $X$ is a parabolic bundle $E_*$ on $X$ together with a Higgs field $\phi : E_* \to E_* \otimes K(D)$, where $K$ is the canonical bundle on $X$. The moduli space $\mathcal{M}_{\mathrm{parH}}$ of parabolic Higgs bundles was constructed by Yokogawa \cite{Y93}.

The Picard group of the moduli space $\mathcal{M}$ of vector bundles over a curve is isomorphic to $\mathbb{Z}$ \cite{ND89}. The Picard group of the moduli space of quasi-parabolic $G$-bundles ($G$ is a complex simple and simply connected algebraic group) over a curve was computed by Laszlo and Sorger in \cite{LS97}. In this paper we study the Picard group of the moduli space of parabolic bundles and the moduli space of parabolic Higgs bundles over $X$. Using the Lemma \ref{keylemma} and the following theorem
\begin{theorem}
 There is an isomorphism
 \[
 \mathrm{Pic}(\mathcal{M}_{\mathrm{par}}) \cong \mathrm{Pic}(\mathcal{M}_{\mathrm{parH}}),
 \]
\end{theorem}
we describe the Picard groups.

Finally we remark that the Picard group of the moduli space of parabolic $\mathrm{SL}(n,\mathbb{C})$-Higgs bundles can be described in a similar way.

\section{Preliminaries}
\subsection{Parabolic bundles}
Let $D = \{p_1,\dots , p_n\} \subset X$ be a set of $n$ distinct marked points. We will fix this set $D$ throughout this article.
\begin{definition}\label{parabolic}
	A \textit{parabolic bundle} $E_*$ of rank $r$ on $X$ is a holomorphic vector bundle $E$ of rank $r$ on $X$ together with a parabolic structure along the divisor $D$, i.e. for each point $p \in D$, we have
	\begin{enumerate}
		\item a filtration of subspaces of the fiber 
		\[
		E_p \eqqcolon E_{p,1}\supsetneq E_{p,2} \supsetneq \dots \supsetneq E_{p,r_p} \supsetneq E_{p,r_p+1} =\{0\},
		\]
		\item a sequence of real numbers (parabolic weights) satisfying 
		\[
		0\leq \alpha_1(p) < \alpha_2(p) < \dots < \alpha_{r_p}(p) < 1,
		\]
	\end{enumerate}
	where $r_p$ is an integer between $1$ and $r$. 
\end{definition}

 A \textit{quasi-parabolic structure} on $E$ consists of flags (without the parabolic weights) over $D$. The collection of all such weights will be denoted by $\alpha =\{(\alpha_1(p),\alpha_2(p),\dots ,\alpha_{r_p}(p))\}_{p\in D}$ for a fixed parabolic structure. The parabolic structure $\alpha$ is said to have \textit{full flags} if 
\[\mathrm{dim}(E_{p,i}/E_{p,i+1}) = 1
\] 
for all $i \in \{1,\dots, r_p\}$ and for all $p\in D$, or equivalently $r_p=r$ for all $p\in D$.

The \textit{parabolic degree} of a parabolic bundle $E_*$ is defined as
\[
\operatorname{pardeg}(E_*) \coloneqq \deg(E)+ \sum\limits_{p\in D}\sum\limits_{i=1}^{r_p} \alpha_i(p) \cdot \dim(E_{p,i}/E_{p,i+1})
\]
and the \textit{parabolic slope} is defined as
\[
\mu_{\mathrm{par}}(E_*) \coloneqq \frac{\text{pardeg}(E_*)}{\mathrm{rk}(E)}.
\]

There is a natural way to define the notion of \textit{dual} and \textit{tensor product} of parabolic bundles (see \cite{Y95}).

\begin{definition}
	A \textit{parabolic subbundle} $F_*$ of $E_*$ is a subbundle $F\subset E$ of the underlying vector bundle endowed with an induced parabolic structure. The parabolic structure on $F$ is defined as follows : for every point $p\in D$, the quasi-parabolic structure on $F$, i.e. the filtration in $F_p$ is given by 
	\[
	F_p \eqqcolon F_{p,1}\supsetneq F_{p,2} \supsetneq \dots \supsetneq F_{p,r'_p} \supsetneq \{0\},
	\]
	where $F_{p,i}= F_p \cap E_{p,i}$, i.e. we are considering the intersection with the already given filtration in $E_p$, and also scrapping all the repetitions of subspaces in the filtration. The parabolic weights $0\leq \alpha'_1(p) < \alpha'_2(p) < \dots < \alpha'_{r'_p}(p) < 1$ are taken to be the largest possible among the given parabolic weights which are allowed after the intersections, i.e. 
	\[
	\alpha'_i(p) = \mathrm{max}_j\{\alpha_j(p)| F_p \cap E_{p,j}=F_{p,i} \}=\mathrm{max}_j\{\alpha_j(p)| F_{p,i}\subseteq E_{p,j} \}
	\]
	That is to say, the parabolic weight associated to $F_{p,i}$ is the weight $\alpha_j(p)$ such that $F_{p,i}\subseteq E_{p,j}$ but $F_{p,i}\nsubseteq E_{p,j+1}$.
\end{definition}
\begin{definition}
	A parabolic bundle $E_*$ is \textit{stable} (resp. \textit{semistable}) if for every nonzero proper subbundle $F_* \subset E_*$, we have
	\[
	\mu_{\mathrm{par}}(F_*) < \mu_{\mathrm{par}}(E_*) \hspace{0.2cm} (\mathrm{resp. } \hspace{0.2cm} \leq).
	\]
\end{definition}

Let $\mathcal{M}$ (resp. $\mathcal{M}^s$) denote the moduli space semistable (resp. stable) vector bundles of rank $r$ and degree $d$.

Let $\mathcal{M}_{\mathrm{par}}$ denote the moduli space of semistable parabolic bundles of rank $r$, degree $d$ and parabolic structure $\alpha$, which was constructed by Mehta and Seshadri in \cite{MS80}. They also showed that it is a normal projective variety of dimension
\[
r^2(g-1) + 1 + \dfrac{n(r^2-r)}{2}, 
\]
where the last summand is because of the fact that we are considering the full flag parabolic structure over each point of $D$. Let $\mathcal{M}^s_{\mathrm{par}}$ denote the moduli of stable parabolic bundles. Also the stable locus $\mathcal{M}^s_{\mathrm{par}}$ is exactly the nonsingular locus of $\mathcal{M}_{\mathrm{par}}$. By \cite[Proposition $5.2$]{BY99}, the parabolic stablity implies the semistability of the underlying bundle whenever the weights are small enough. In this paper, we will assume that the weights are small enough and the rank $r$ and degree $d$ are coprime.

\subsection{Parabolic Higgs bundles}
Let $K$ denote the canonical bundle on $X$. We write $K(D) \coloneqq K \otimes \mathcal{O}(D)$.
\begin{definition}
	 A \textit{parabolic Higgs bundle} on $X$ is a parabolic bundle $E_*$ on $X$ together with a Higgs field $\Phi : E_* \to E_* \otimes K(D)$ such that $\Phi$ is strongly parabolic, i.e. $\Phi(E_{p,i}) \subset E_{p,i+1} \otimes \left.K(D)\right|_p$ for all $p \in D$. 
	 
\end{definition}
We also have a notion of parabolic Higgs bundle where the Higgs field $\Phi$ is only assumed to be parabolic, i.e. $\Phi(E_{p,i}) \subset E_{p,i} \otimes \left.K(D)\right|_p$ for all $p \in D$. However in this paper we will always assume that the Higgs field $\Phi$ is strongly parabolic.

\begin{definition}
		A parabolic Higgs bundle $(E_*,\Phi)$ is \textit{stable} (resp. \textit{semistable}) if for every nonzero proper $\Phi$-invariant subbundle $F_* \subset E_*$ (i.e. $\Phi(F_*)\subset F_*\otimes K(D)$), we have
			\[
	\mu_{\mathrm{par}}(F_*) < \mu_{\mathrm{par}}(E_*) \hspace{0.2cm} (\mathrm{resp. } \hspace{0.2cm} \leq).
	\]
\end{definition}
The moduli space $\mathcal{M}_{\mathrm{parH}}$ of semistable parabolic Higgs bundles of rank $r$, degree $d$ and parabolic structure $\alpha$ was constructed by Yokogawa in \cite{Y93} (see \cite{BY96} for furthur discussion). It is a normal quasi-projective variety. The moduli space $\mathcal{M}^s_{\mathrm{parH}}$ of stable parabolic Higgs bundles is the nonsingular locus of $\mathcal{M}_{\mathrm{parH}}$.

\section{Picard group of mdouli of parabolic bundles}

\begin{prop}\label{prop}
	Assume that the parabolic weights are small enough. There exist a morphism $f : \mathcal{M}^s_{\mathrm{par}} \to \mathcal{M}^s$ making $\mathcal{M}^s_{\mathrm{par}}$ into a projective bundle over $\mathcal{M}^s$.
\end{prop}
\begin{proof}
	Since the weights are small enough and the rank and degree are coprime, the parabolic stability is equivalent to the stability of the underlying vector bundle. Therefore, we have the forgetful morphism
	\[
	f : \mathcal{M}^s_{\mathrm{par}} \to \mathcal{M}^s,
	\]
which sends a stable parabolic bundle to the underlying vector bundle. 

For simplicity first assume that $n=|D| = 1$, i.e. we have only one parabolic point $p \in X$. For each $[E] \in \mathcal{M}^s$, the fiber of $f$ is the flag variety of $E_p$ (the fiber of $E$ over $p$).
	
	In general, if $D = \{p_1,\dots, p_n\}$ then for each point $[E] \in \mathcal{M}^s$ the fiber of $f$ is isomorphic to the product of flag varieties of the vector spaces $E_{p_i}$ ($1\leq i \leq n$).	
\end{proof}	

\begin{lemma}\label{keylemma}
	There exists a short exact sequence
	\[
	1 \to \mathrm{Pic}(\mathcal{M}) \to \mathrm{Pic}(\mathcal{M}_{\mathrm{par}}) \to \mathbb{Z} \to 1.
	\]
\end{lemma}
\begin{proof}
	The first morphism $\mathrm{Pic}(\mathcal{M}) \to \mathrm{Pic}(\mathcal{M}_{\mathrm{par}})$ is actually the pull-back of line bundles over $\mathcal{M}$. By \ref{prop}, $\mathcal{M}^s_{\mathrm{par}}$ is a projective bundle over $\mathcal{M}^s$. Therefore, we have a short exact sequence
	\[
	1 \to \mathrm{Pic}(\mathcal{M}^s) \to \mathrm{Pic}(\mathcal{M}^s_{\mathrm{par}}) \to \mathbb{Z} \to 1
	\]
	where the second morphism is the restriction to the generic fiber. Since the rank and degree are coprime, we have $\mathrm{Pic}(\mathcal{M}^s) = \mathrm{Pic}(\mathcal{M})$. Since the moduli space $\mathcal{M}_{\mathrm{par}}$ is a normal variety and the stable locus $\mathcal{M}^s_{\mathrm{par}}$ is exactly the nonsingular locus, the codimension of the complement is at least two. Therefore by \cite[Proposition $1.6$]{H80} $\mathrm{Pic}(\mathcal{M}^s_{\mathrm{par}}) = \mathrm{Pic}(\mathcal{M}_{\mathrm{par}})$, and the lemma follows.
\end{proof}	

\begin{remark}\label{remark}
Since $\mathrm{Pic}(\mathcal{M}) \cong \mathbb{Z}$ (see \cite{ND89}), the above Lemma \ref{keylemma} calculates $\mathrm{Pic}(\mathcal{M}_{\mathrm{par}})$. 
\end{remark}

\section{Picard group of moduli of parabolic Higgs bundles}
Let $V$ be a smooth quasi-projective complex variety and let $q : E \to V$ be a vector bundle over $V$ of finite rank. The following lemma is an well-known result.
\begin{lemma}\label{lemma}
    The induced homomorphism $q^* : \mathrm{Pic}(V) \to \mathrm{Pic}(E)$ is an isomorphism.
\end{lemma}
 \begin{proof}
     Let $E$ be a rank $n$ vector bundle over $V$. Since the fibers of $q$ are affine $n$-spaces, we have the the short exact sequence
     \[
     1 \to \mathrm{Pic}(V) \overset{q^*}{\to} \mathrm{Pic}(E) \to \mathrm{Pic}(\mathbb{A}^n_{\mathbb{C}}) \to 1.
     \]
     Since $\mathrm{Pic}(\mathbb{A}^n_{\mathbb{C}})$ is trivial, $q^*$ is an isomorphism.
 \end{proof} 
 Similar to the non-parabolic situation, we also have the Serre duality for the parabolic bundles (see \cite{Y95}, \cite{BY96}). If $E_* \in \mathcal{M}^s_{\mathrm{par}}$ then by the Serre duality
 \[
 T^*_{E_*}\mathcal{M}^s_{\mathrm{par}} \cong H^1(X,\mathrm{PEnd}(E_*))^{\vee} \cong H^0(X,\mathrm{SPEnd}(E_*)\otimes K(D)),
 \]
 where $\mathrm{PEnd}(E_*)$ and $\mathrm{SPEnd}(E_*)$ are the parabolic and strongly parabolic endomorphisms of $E_*$ respectively. Thus the cotangent bundle $T^*\mathcal{M}^s_{\mathrm{par}}$ is an open subvariety of $\mathcal{M}^s_{\mathrm{parH}}$. 
 \begin{prop}\label{prop}
     Let $g \geq 3$. Then the complement $\mathcal{M}^s_{\mathrm{parH}} \setminus T^*\mathcal{M}^s_{\mathrm{par}}$ has codimension at least $2$.
 \end{prop}
 \begin{proof}
     This follows from the work of Faltings \cite[Theorem II.$6$ (iii)]{F93}.
 \end{proof}
 \begin{theorem}\label{thm2}
 There is an isomorphism
 \[
 \mathrm{Pic}(\mathcal{M}^s_{\mathrm{par}}) \cong \mathrm{Pic}(\mathcal{M}^s_{\mathrm{parH}}).
 \]
 \end{theorem}
 \begin{proof}
     Since $T^*\mathcal{M}^s_{\mathrm{par}}$ is a vector bundle of finite rank over $\mathcal{M}^s_{\mathrm{par}}$, by Lemma \ref{lemma} we have
     \[
     \mathrm{Pic}(\mathcal{M}^s_{\mathrm{par}}) \cong \mathrm{Pic}(T^*\mathcal{M}^s_{\mathrm{par}}).
     \]
     Then by Proposition \ref{prop}, we have the isomorphism
     \[
     \mathrm{Pic}(T^*\mathcal{M}^s_{\mathrm{par}}) \cong \mathrm{Pic}(\mathcal{M}^s_{\mathrm{parH}}).
     \]
     Hence the theorem follows.
 \end{proof}
 \begin{corollary}\label{cor}
 There is an isomorphism
 \[
 \mathrm{Pic}(\mathcal{M}_{\mathrm{par}}) \cong \mathrm{Pic}(\mathcal{M}_{\mathrm{parH}}).
 \]
 \end{corollary}
 \begin{proof}
     We know that 
     \[
     \mathrm{Pic}(\mathcal{M}_{\mathrm{par}}) \cong \mathrm{Pic}(\mathcal{M}^s_{\mathrm{par}}).
     \]
     Similarly, since $\mathcal{M}_{\mathrm{parH}}$ is a normal variety and $\mathcal{M}^s_{\mathrm{parH}}$ is the nonsingular locus, we have 
     \[
          \mathrm{Pic}(\mathcal{M}^s_{\mathrm{parH}}) \cong \mathrm{Pic}(\mathcal{M}_{\mathrm{parH}}).
     \]
     Therefore the corollary follows from Theorem \ref{thm2}.
 \end{proof}
 \begin{remark}
 Thus the Remark \ref{remark} and the Corollary \ref{cor} together calculates the Picard group $\mathrm{Pic}(\mathcal{M}_{\mathrm{parH}})$.
 \end{remark}
 \begin{remark}
 Similarly, the Picard group of the moduli space of parabolic $SL(n,\mathbb{C})$-Higgs bundles (i.e. parabolic Higgs bundles with trivial determinant) can be calculated using the same method as above. 
 \end{remark}

 \section*{Acknowledgement}
 This work was supported by the Institute for Basic Science (IBS-R003-D1).

\end{document}